\documentclass[12pt]{amsart}

\usepackage{amsthm}
\usepackage{amsmath}
\usepackage{amssymb}
\usepackage{graphicx}
\usepackage{enumerate}
\usepackage{xcolor}
\usepackage{bbm}
\usepackage{hyperref}
\usepackage{mathrsfs}

\allowdisplaybreaks

\numberwithin{equation}{section}
\numberwithin{figure}{section}

\theoremstyle{plain}
\newtheorem{theorem}{\sffamily Theorem}[section]
\newtheorem{proposition}[theorem]{\sffamily Proposition}
\newtheorem{lemma}[theorem]{\sffamily Lemma}
\newtheorem{corollary}[theorem]{\sffamily Corollary}
\newtheorem{example}[theorem]{\sffamily Example}
\newtheorem{remark}[theorem]{\sffamily Remark}
\newtheorem{definition}[theorem]{\sffamily Definition}

\def\BET{\begin{theorem}}
\def\ENT{\end{theorem}}
\def\BEP{\begin{proposition}}
\def\ENP{\end{proposition}}
\def\BEL{\begin{lemma}}
\def\ENL{\end{lemma}}
\def\BEC{\begin{corollary}}
\def\ENC{\end{corollary}}
\def\BEE{\begin{example}\rm}
\def\ENE{\end{example}}
\def\BER{\begin{remark} \rm}
\def\ENR{\end{remark}}
\def\BED{\begin{definition} \rm}
\def\END{\end{definition}}

\def\bea{\begin{eqnarray}}
\def\eea{\end{eqnarray}}

\def\bean{\begin{eqnarray*}}
\def\eean{\end{eqnarray*}}

\def\beq{\begin{equation}}
\def\eeq{\end{equation}}

\def\beal{\begin{align*}}

\def\eeal{ \end{align*} }

\def\rowpl{\nonumber \\ & \ \ &+  }

\def\bbB{{\mathbb B}}
\def\bbC{{\mathbb C}}
\def\bbD{{\mathbb D}}

\def\bbN{{\mathbb N}}

\def\bbQ{{\mathbb Q}}
\def\bbR{{\mathbb R}}

\def\cJ{{\mathcal J}}
\def\cK{{\mathcal K}}

\def\cQ{{\mathcal Q}}

\DeclareMathOperator{\supp}{supp}
\DeclareMathOperator{\diam}{diam}

\begin{document}

\title[Toeplitz operators on the unit ball]{Toeplitz operators on the unit ball with locally integrable symbols}

\author{R. Hagger}
\address{Department of Mathematics, University of Reading, England}
\email{r.t.hagger@reading.ac.uk}

\author{C. Liu}
\address{School of Mathematical Sciences, University of Science and Technology of China, Hefei, Anhui 230026, People's Republic of China}
\email{cwliu@ustc.edu.cn}

\author{J. Taskinen}
\address{Department of Mathematics and Statistics, University of Helsinki, Finland}
\email{jari.taskinen@helsinki.fi}

\author{J. A. Virtanen}
\address{Department of Mathematics, University of Reading, England}
\email{j.a.virtanen@reading.ac.uk}

\thanks{The first author was supported by the European Union’s Horizon 2020 research and innovation programme under the Marie Sklodowska-Curie grant agreement No 844451. The fourth author was supported in part by Engineering and Physical Sciences Research
Council (EPSRC) grant EP/T008636/1}


\begin{abstract}
We study the boundedness of Toeplitz operators $T_\psi$ with locally integrable
symbols on weighted harmonic Bergman spaces over the unit ball of $\bbR^n$.
Generalizing earlier results for analytic function spaces, we derive a general
sufficient condition for the boundedness of $T_\psi$ in terms of
suitable averages of its symbol. We also obtain a similar ``vanishing'' condition for compactness. Finally, we show how these results can be transferred to the setting of the standard weighted Bergman spaces of analytic functions.
\end{abstract}

\maketitle

\section{Introduction} \label{sec1.1}
Denote by $\mathrm{d}V$ the normalized $n$-dimensional Lebesgue measure on the unit ball $\bbB_n$ of $\bbR^n$ with $n \geq 2$. For $\lambda>-1$ and $p \geq 1$, let $b_\lambda^p = b_\lambda^p (\bbB_n)$ be the harmonic Bergman space consisting of all harmonic complex-valued functions in $L_\lambda^p=L^p(\bbB_n, \mathrm{d}V_\lambda)$, where $\mathrm{d}V_\lambda = c(n,\lambda) (1-|x|^2)^\lambda \mathrm{d}V$ and $c(n,\lambda)$ is a normalization constant (see \eqref{1.10}). We further write $b^p$ for $b_0^p$ and $L^p$ for $L^p_0$.

The Toeplitz operator $T_\psi$ with symbol $\psi$ is defined by
\bea
    T_\psi f = P_\lambda(\psi f),    \label{0.1}
\eea
where $P_\lambda$ is the orthogonal projection of $L_\lambda^2$ onto $b_\lambda^2$ (see  \eqref{1.15}), $\psi$ is a measurable  function on $\bbB_n$, and $f$ is harmonic in $\bbB_n$. It is known that $P_\lambda$ can be extended to a bounded projection from $L_\lambda^p$ onto $b_\lambda^p$ for $1<p<\infty$ (see \cite{JP}). It follows that $T_\psi$ is well defined and
bounded on  $b_\lambda^p$ whenever $\psi$ is bounded. For $f \in L_\lambda^1$, $P_{\lambda}f$ is still well defined as a function (see \eqref{1.15} below), but we may have $P_{\lambda}f \notin L_\lambda^1$ in general.

The study of $T_\psi$ on $b^p$ goes back to \cite{Miao2}, where it was proven for non-negative symbols $\psi \in L^1$ that $T_\psi$ is bounded (compact) on $b^p$ if and only if the averaging function
\bea
x \mapsto \frac{1}{|E_r(x)|}\int_{E_r(x)} \psi \, \mathrm{d}V  \label{1.1a}
\eea
is bounded on $\bbB_n$ (vanishes as $|x|\to 1$), where $E_r(x)=\{y\in\bbB_n : |y-x|<r(1-|x|)\}$ with $r \in (0,1)$ and $|E_r(x)|$ denotes its volume. Further, it was shown that for symbols $\psi$ continuous on $\overline\bbB_n$, $T_\psi$ is compact on $b^p$ if and only if $\psi=0$ on $\partial\bbB_n$, which was generalized to $b^2_\lambda$ in \cite{S1998}. It is easy to see that these conditions for boundedness (compactness) can also be formulated in terms of boundedness (vanishing) of the Berezin transform of $\psi$ or variants of the averaging function \eqref{1.1a}. Moreover, applying \eqref{1.1a} to the positive and negative parts of ${\rm Re}\, \psi $ and ${\rm Im}\, \psi$ shows that if
\begin{equation} \label{1.229}
\sup_{x \in \bbB_{n}} \frac{1}{|E_r(x)|}\int_{E_r(x)} |\psi| \, \mathrm{d}V < \infty,
\end{equation}
then $T_{\psi}$ is bounded. However, the modulus in the integrand makes \eqref{1.229} far from being necessary if the symbol is not positive. In this paper we propose averages over different sets, certain spherical boxes, but also with the modulus outside of the integral; see \eqref{eqn:setfunction}. In particular, we give a new, weaker sufficient condition for the boundedness and compactness of Toeplitz operators on the weighted harmonic Bergman spaces $b_\lambda^p$ in Theorems \ref{thm:mainthm} and
\ref{thm:mainthm2}. In Corollary \ref{main-cor}
we present the corresponding results in the case of weighted Bergman spaces
$A^p_\lambda$ of analytic functions on the unit ball of $\bbC^n$.

Most results about Toeplitz operators on $b_\lambda^p$ are generalizations from the
setting of analytic Bergman spaces $A^p_\lambda$ using similar but also additional
ideas related to non-analyticity. As in $b_\lambda^p$, characterizing bounded  Toeplitz operators is an open problem in $A^p(\bbD)$ (the unweighted Berg\-man space of the
unit disk $\bbD$ of the complex plane $\bbC$) even in the case
$p=2$. Compared with \eqref{1.229}, much more general and weaker sufficient conditions were found for Toeplitz operators on $A^p(\bbD)$ in \cite{TV2} with further improvements in \cite{TV2018}. Roughly speaking, the conditions resemble \eqref{1.229}, but the modulus appears outside the integral. This means that the modulus of a (wildly oscillating) symbol may be very large, but it may still induce a bounded operator due to cancellation in the integral; for an illustrative example, see~\cite{TV2018}. Recently, the same idea was used to introduce weak BMO and VMO type conditions in~\cite{arXiv:2106.11734}.
Here, we prove similar results for the space $b_\lambda^p$ of the unit ball with standard weights for the sake of maximal generality, which carries a number of technical challenges. We note that recently in~\cite{YZ2019} it was shown that the mentioned sufficient conditions in $A^p$  are not necessary for the boundedness. We also refer to \cite{YZ2019} for a concise account of the study of boundedness of Toeplitz operators on Bergman spaces.

As for the contents of this paper, the main results are formulated in
Section \ref{sec2Main} and their proofs are prepared in Sections \ref{sec1.2} and
\ref{sec2}, which contain some more necessary notation, definitions
and preliminary lemmas. The proofs of the main results are completed in Section
\ref{sec3}, and in Section \ref{sec9} we construct examples of radial, oscillating
symbols, the modulus of which may grow arbitrarily fast, but which still induce 
bounded Toeplitz operators.

\section{The main results}
\label{sec2Main}
In what follows we generalize the results of \cite{TV2, TV2018} from the case of
analytic  functions in $\bbD$ to the case of harmonic functions in $\bbB_n$ and also
consider weighted norms. Our analogous sufficient condition is a rather weak requirement of the boundedness of certain averages of $\psi$ over spherical boxes (see Theorem \ref{thm:mainthm}). We present the main results, Theorems~\ref{thm:mainthm}--\ref{thm:mainthm2} and Corollary~\ref{main-cor}, with a minimal amount of notation.

It seems likely that analogous results hold for little Hankel operators on $b_\lambda^p$, too, but we do not consider this question here; cf.~\cite{TV2018}.

Let $\bbQ_n := [0,1) \times [0,\pi]^{n-2} \times [0,2\pi]$ and $\sigma : \bbQ_n \to \bbB_n$ be defined by
\begin{align*}
&\sigma(r,\theta_2,\ldots,\theta_{n-1},\theta_n)\\
&\qquad \qquad \qquad = (r\cos\theta_2,r\sin\theta_2\cos\theta_3,r\sin\theta_2\sin\theta_3\cos\theta_4,\ldots,\\
&\qquad \qquad \qquad \qquad r\sin\theta_2\cdots\sin\theta_{n-1}\cos\theta_n,r\sin\theta_2\cdots\sin\theta_{n-1}\sin\theta_n).
\end{align*}
We note that $\sigma$ is surjective, and injective almost everywhere (i.e.~for almost every $x \in \bbQ_n$, $\sigma(x) = \sigma(y)$ implies $x = y$). For $x,y \in \bbQ_n$ we define the boxes
\[Q(x,y) := \big\{z \in \bbQ_n : z_j \in [\min\{x_j,y_j\},\max\{x_j,y_j\}]\big\}\]
and $B(x,y) := \sigma(Q(x,y))$. In order to distinguish between $B(x,y) \subset \bbB_n$ and $Q(x,y) \subset \bbQ_n$, we call $B(x,y)$ a spherical box. For the following, we need to single out certain spherical boxes. Let
\begin{align*}
\cK &:= \Big\{(m,k_2,k_3,\ldots,k_n) \in \mathbb{N}_0^n : 0 \leq k_n \leq k_{n-1} \leq \ldots \leq k_2 \leq 2^m-1\Big\}.
\end{align*}
For $k \in \cK$ we define
\begin{align*}
C_k &:= [1-2^{-m},1-2^{-m-1}] \times [\tfrac{\pi}{2} k_2 2^{-m},\tfrac{\pi}{2} (k_2+1) 2^{-m}] \times [\tfrac{\pi}{2}\tfrac{k_3}{k_2+1},\tfrac{\pi}{2}\tfrac{k_3+1}{k_2+1}]\\
&\qquad \times \ldots \times [\tfrac{\pi}{2}\tfrac{k_{n-1}}{k_{n-2}+1},\tfrac{\pi}{2}\tfrac{k_{n-1}+1}{k_{n-2}+1}] \times [2\pi\tfrac{k_n}{k_{n-1}+1},2\pi\tfrac{k_n+1}{k_{n-1}+1}].
\end{align*}
Note that $\bigcup\limits_{k \in \cK} C_k = [0,1) \times [0,\tfrac{\pi}{2}]^{n-2} \times [0,2\pi]$. If we add all the sets that can be obtained from the sets $C_k$ by repeatedly using the reflections $\theta_2 \mapsto \pi - \theta_2, \ldots, \theta_{n-1} \mapsto \pi - \theta_{n-1}$, we obtain a cover of $\bbQ_n$. Let $\cQ$ denote the collection of these sets. Enumerate the elements of $\cQ$ by $Q_1,Q_2,\ldots$ and define $B_j := \sigma(Q_j)$ for $j \in \bbN$. Then $\bigcup\limits_{j \in \bbN} B_j$ covers $\bbB_n$ and $B_j \cap B_k$ is a null set whenever $j \neq k$. If $|x| \in [1-2^{-m},1-2^{-m-1}]$ for all $x \in B_j$, then $B_j$ is called a dyadic box of generation $m$. For a dyadic box $B_j$ of generation $m$, we denote $B_j^* := B_j + 2^{-m-2}\bbB_n$. The collection of these spherical boxes $B_j$ has the property that they are all of size comparable to $2^{-mn}$ and there is a constant $N \in \bbN$ such that every $x \in \bbB_n$ is contained in at most $N$ of the sets $B_j^*$. We will prove these facts in Lemma \ref{lem:propertiesofQ} below.

We remark that it may be tempting and more natural to extend the decomposition of the unit disk in \cite{TV2} to higher dimensions using boxes of the form
\[C_k = [1-2^{-m},1-2^{-m-1}] \times \prod\limits_{j = 2}^n [2^{-m}k_j\pi,2^{-m}(k_j+1)\pi] \subset \bbQ_n,\]
where $k_j = 1,\ldots,2^m-1$ for $j = 2,\ldots,n-1$ and $k_n = 1,\ldots,2^{m+1}-1$. After mapping them to $\bbB_n$ and some reordering, these boxes also provide a countable, pairwise essentially disjoint cover $\{B_j : j \in \bbN\}$ of $\bbB_n$. However, it turns out that the size of these boxes is not always comparable to $2^{-mn}$. This is due to the fact that the Jacobian of $\sigma$ is singular if any of $\theta_2, \ldots, \theta_{n-1}$ is equal to $0$ or $\pi$. It also means that too many of the sets $B_j^*$ overlap close to the singular points. As these two properties are crucial for our analysis, we had to be more careful in the definition of the sets $B_j$. Having said that, it is highly expected that many other decompositions also lead to the same result. Our choice is not special in that regard, but it is somewhat natural as it leads to a definition of generalized Toeplitz operators that is independent of the decomposition; cf.~Theorem \ref{limit-thm}.

The set of locally integrable functions on $\bbB_n$ will be denoted by $L_{\rm loc}^1$. Note that $L_{\rm loc}^1$ does not depend on $\lambda$. The characteristic function of a measurable set $E \subset \bbB_n$ will be denoted by $\chi_E$.

\BED
\label{def:generalizedToeplitz}
Let $\psi \in L_{\rm loc}^1$, $1 < p < \infty$ and assume that the series
\begin{equation} \label{eqn:generalizedToeplitz}
T_\psi f(x) := \sum_{j=1}^\infty  T_\psi(\chi_{B_j}f)(x) = \sum_{j=1}^\infty  P_\lambda(\psi\chi_{B_j}f)(x)
\end{equation}
converges for almost every $x \in \bbB_n$ and all $f \in b_\lambda^p$. Then $T_\psi$ is called a \textit{generalized Toeplitz operator}.
\END

As $P_{\lambda}$ is bounded on $L_\lambda^p$, it is clear that $T_{\psi}f = P_{\lambda}(\psi f)$ whenever $\psi f \in L_\lambda^p$. In particular, if $\psi$ is bounded, then $T_{\psi}$ is just the usual Toeplitz operator as defined in the introduction. However, without any further assumptions it is a priori not even clear if $T_{\psi}f \in b_\lambda^p$. Our first main result is that $T_{\psi}$ is a well-defined bounded linear operator under the ``weak'' Carleson-type condition \eqref{eqn:weakCarleson}.

For any measurable set $B \subset \bbB_n$ we denote the weighted volume
by $|B|_\lambda := \int_B \, \mathrm{d}V_\lambda$. Moreover, $|B| := |B|_0$.

Let $x^{(j)}$ and $y^{(j)}$ denote the smallest, respectively largest, element of $Q_j$ with respect to the partial order
\begin{align} \label{eq:order}
x \leq y \Longleftrightarrow \, &x_1 \leq y_1, \, |\tfrac{\pi}{2} - x_2| \geq |\tfrac{\pi}{2} - y_2|, \, \ldots, \, |\tfrac{\pi}{2} - x_{n-1}| \geq |\tfrac{\pi}{2} - y_{n-1}|,\notag\\
&x_n \leq y_n
\end{align}
imposed on $\bbQ_n$.

\begin{remark}
Note that for $x,y \in [0,1) \times [0,\tfrac{\pi}{2}]^{n-2} \times [0,2\pi]$ this is just the usual partial order of points in $\bbR^n$, which is then mirrored to all of $\bbQ_n$ to account for the construction of the sets $Q_j$ and $B_j$. In particular, the $x^{(j)}$ and $y^{(j)}$ are two opposite corners of $Q_j$ and we have $B_j = B(x^{(j)},y^{(j)})$. In the following, we will only work in the subset $[0,1) \times [0,\tfrac{\pi}{2}]^{n-2} \times [0,2\pi]$ and then use that everything is symmetric around $\theta_2 \mapsto \pi - \theta_2, \ldots, \theta_{n-1} \mapsto \pi - \theta_{n-1}$. This simplifies the notation significantly and also means that we do not have to worry about the slightly unusual order \eqref{eq:order}.
\end{remark}

Given a function $\psi \in L_{\rm loc}^1$, we define
\begin{equation}\label{eqn:setfunction}
\widehat{\psi}_j :=  \sup_{y \in B_j} \bigg|\int\limits_{B(x^{(j)},y) } \psi \, \mathrm{d}V_{\lambda}\bigg| .
\end{equation}

\BET\label{thm:mainthm}
Let $\psi \in L_{\rm loc}^1$, $1<p<\infty$ and the family $(B_j)_{j \in \bbN}$ be as above. If there exists a constant $C_\psi \geq 0$ such that
\begin{equation}\label{eqn:weakCarleson}
\widehat{\psi}_j \leq C_\psi |B_j|_{\lambda}
\end{equation}
for all $j \in \bbN$, then the series \eqref{eqn:generalizedToeplitz} converges almost everywhere and in $L_\lambda^p$. Moreover, $T_\psi$ defines a bounded linear operator on $b_{\lambda}^p$ and there is a constant $C \geq 0$ independent of $\psi$ such that $\|T_{\psi}\| \leq CC_{\psi}$.
\ENT

For $\psi\in L_{\rm loc}^1$ and $0<\rho<1$, we define $\psi_\rho(z) = \psi(z)$ for $|z|\le \rho$ and $\psi_\rho(z) = 0$ for $\rho<|z|<1$. Note that $T_{\psi_\rho}$ is bounded on $b_\lambda^p$ by the previous theorem (see also Lemma \ref{lem:L1boundedness} below). Our next theorem gives an alternative definition of Toeplitz operators with $L^1_{\rm loc}$-symbol that does not depend on the decomposition $(B_j)_{j \in \bbN}$.

\begin{theorem}\label{limit-thm}
Let $1<p<\infty$ and $1/p+1/q=1$, and suppose that $\psi\in L_{\rm loc}^1$ satisfies \eqref{eqn:weakCarleson}. Then
\[T_\psi f = \lim_{\rho\to 1} T_{\psi_\rho} f\]
for all $f\in b_\lambda^p$ and $T_{\psi}^* : b_\lambda^q\to b_\lambda^q$ can be expressed as
$$
	T_{\psi}^*f = \lim_{\rho\to 1} T_{\overline \psi_\rho} f
$$
for $f\in b_\lambda^q$.
 \end{theorem}

Here, $T_{\psi}^*$ denotes the adjoint of the operator $T_{\psi}$ with respect to the standard duality of $b_\lambda^p$-spaces.

Concerning the compactness of the Toeplitz operator, we formulate the result in terms of the corresponding vanishing weak Carleson condition.

\BET\label{thm:mainthm2} Let the symbol $\psi$ satisfy the assumptions of Theorem \ref{thm:mainthm} and in addition the condition
\begin{equation}\label{eqn:weakvanishingCarleson}
\lim_{j \to \infty} \frac {\widehat{\psi}_j}{|B_j|_{\lambda}}=0.
\end{equation}
Then $T_{\psi}:b_{\lambda}^p\to b_{\lambda}^p$ is compact for all  $1<p<\infty$.
\ENT

We emphasize that these sufficient conditions do not concern the modulus of the symbol $\psi$, only that of the integral.

Finally, we consider Toeplitz operators, denoted by $T_\psi^{\rm an}$, in the case of the weighted
Bergman spaces $A_\lambda^p= A_\lambda^p ({\bf B}_n)$ of the open unit ball ${\bf B}_n$ of $\bbC^n$. Here, $\bbC^n $ is identified in the canonical way with $\bbR^{2n}$,
hence ${\bf B}_n = \bbB_{2n}$, and the Bergman space $A_\lambda^p
$ is the closed subspace of $L_\lambda^p(\bbB_{2n})$
consisting of analytic functions of $n$ complex  variables in the domain
${\bf B}_n$. The Toeplitz operator $T_\psi^{\rm an}$ with symbol $\psi \in L_{\rm loc}^1(\bbB_{ 2n})$
is defined as in \eqref{0.1} by replacing the projection $P_\lambda$ by the
orthogonal projection $P_\lambda^{\rm an}$ from $L_\lambda^p(\bbB_{2n})$ onto
$A_\lambda^p ({\bf B}_n)$, and the generalized Toeplitz operator corresponding
to \eqref{eqn:generalizedToeplitz} is defined analogously.

\begin{corollary}\label{main-cor}
Let the assumptions of Theorem \ref{thm:mainthm}  (respectively, Theorem
\ref{thm:mainthm2}) be valid for the symbol $\psi$  in the domain $\bbB_{2n}$, $n \in \bbN$.
Then,  the Toeplitz operator
$T_\psi^{\rm an}: A_\lambda^p({\bf B}_n) \to A_\lambda^p({\bf B}_n)$ is bounded
(resp. compact). The  statement corresponding to Theorem \ref{limit-thm}
is also valid in the space  $A_\lambda^p({\bf B}_n)$.
\end{corollary}

The proofs of the main results will be given in Section \ref{sec3}.

\section{Further notation and definitions}\label{sec1.2}
By $C$, $c$, $C'$ etc.~we mean  positive constants
which may vary from place to place. If the constant depends on some
parameter, say $n$, this is shown as $C(n)$.
For positive valued expressions $f$ and $g$ depending on some variables or
parameters, the notation $f \approx g$ (respectively,
$f \lesssim g$) means the existence of
constants $c, C $ such that $c f \leq g \leq Cf$ (resp. $ f \leq C g$ )
for all values of the
variables or parameters.
By $E(a,r)\subset \bbR^n$ we denote the open Euclidean ball with center $a \in \bbR^n$ and radius $r>0$.

On $\bbB_n$ we define the standard weight function
\bea
\label{1.4}
w(x) = 1 - |x|^2 .
\eea
Given $\lambda>-1$
and $1\leq p<\infty$, the norm of the weighted space $L_\lambda^p$
and harmonic Bergman space $b_{\lambda}^p \subset L_\lambda^p$
is defined by
\[
\|f\|_{p,\lambda}:= \left(\int_{\bbB_n} |f|^p \, \mathrm{d}V_{\lambda}\right)^{1/p} ,
\]
where
\begin{equation}
\label{1.10}
\mathrm{d}V_{\lambda}
:= c(n,\lambda) w^\lambda \, \mathrm{d}V
:=  \frac {2}{n}\cdot \frac {\Gamma(n/2+\lambda+1)}{\Gamma(n/2)\Gamma(\lambda+1)} w^{\lambda} \, \mathrm{d}V  .
\end{equation}
The orthogonal
projection $P_{\lambda}$ from  $L_{\lambda}^2$  onto $b_{\lambda}^2$
(the harmonic Bergman projection)
can be expressed as an integral operator
\bea
P_{\lambda}f(x)=\int\limits_{\bbB_n} f(y) R_{\lambda}(x,y) \, \mathrm{d}V_{\lambda}(y),
\label{1.15}
\eea
where, for $x,y\in \bbB_n$,
\[
R_{\lambda}(x,y)=\frac {\Gamma(n/2)}{\Gamma(n/2+\lambda+1)} \sum_{k=0}^{\infty} \frac {\Gamma(k+n/2+\lambda+1)}{\Gamma(k+n/2)}Z_k(x,y)
\]
and $Z_k (x,y)$ denote the extended zonal harmonics of order $k$. We refer to \cite[Chapter 5]{ABR}
for the definition of these functions  and related facts. This series converges absolutely and uniformly on $K \times \overline{\bbB}_n$ for every compact set $K\subset {\bbB_n}$; see \cite[Proposition 2.6]{NT}. In particular, $R_{\lambda}$ is a smooth bounded function on $K \times \bbB_n$ and also on $\bbB_n \times K$ by symmetry. For $1 < p < \infty$ the boundedness of $P_\lambda : L_{\lambda}^p \to b_{\lambda}^p$, defined as in \eqref{1.15}, is proven in \cite[Theorem 3.1]{JP}.

The following pointwise estimate follows directly from the mean value property of harmonic functions.

\BEL\label{lem:pointwiseest}
Let $\lambda>-1$, $1\leq p<\infty$, and $f\in b_{\lambda}^p$. Then
\[
|f(x)| ~\lesssim~ \frac {\|f\|_{p,\lambda}}{w(x)^{(n+\lambda)/p}}
\]
for all $x\in \bbB_n$.
\ENL

For compactly supported symbols we have the following result, which is well known to experts; however, we include its proof for completeness because we do not know a reference.

\BEL
\label{lem:L1boundedness}
Let $\psi \in L_\lambda^1$ have compact support, that is, $\supp \psi \subseteq r\overline{\bbB_n}$ for some $r < 1$. Then, $P_{\lambda}\psi$ is harmonic and there is a constant $C = C(r)$ such that
\[\|P_{\lambda}\psi\|_{p,\lambda} \leq C\|\psi\|_{1,\lambda}\]
for all $1 < p < \infty$. In particular, $T_\psi$ is well-defined and bounded on $b_{\lambda}^p$.
\ENL

\begin{proof}
We have
\[|P_{\lambda}\psi(x)| \leq \int\limits_{r\overline{\bbB}_n} |\psi(y)||R_{\lambda}(x,y)| \, \mathrm{d}V_{\lambda}(y) \leq C\int\limits_{r\overline{\bbB}_n} |\psi(y)| \, \mathrm{d}V_{\lambda}(y) = C\|\psi\|_{1,\lambda}\]
for all $x \in \bbB_n$ because $R_{\lambda}(x,y)$ is bounded on $\bbB_n \times r\overline{\bbB}_n$ as noted above. This implies that $P_{\lambda}\psi$ is harmonic and $\|P_{\lambda}\psi\|_{p,\lambda} \leq C\|\psi\|_{1,\lambda}$ for all $p$.

By Lemma \ref{lem:pointwiseest}, $|f(y)| \leq \tilde{C}(r) \|f\|_{p,\lambda}$ for $|y| \leq r$. The same estimate as above thus yields the boundedness of $T_\psi$.
\end{proof}

The maximal harmonic Bergman projection $P_\lambda^M$ is the nonlinear operator
\bea
P_\lambda^M  f(x) =   \int\limits_{\bbB_n} |f(y)| |R_{\lambda}(x,y)| \, \mathrm{d}V_{\lambda}(y).\notag
\eea
For $1<p<\infty$ this is a well-defined mapping $L_\lambda^p \to L_\lambda^p$, and it is also bounded in the sense that, for some constant
\bea
\Vert P_\lambda^M f \Vert_{p,\lambda} \leq C \Vert f \Vert_{p,\lambda}
\label{1.15N}
\eea
for $f \in L_\lambda^p$; see again \cite{JP}.

Recall that  for each $a\in {\bbB_n}$, the M\"obius transformation $\varphi_a :{\bbB_n} \to {\bbB_n}$
is defined by the formula
\begin{equation}\label{1.15a}
\varphi_a(x)=\frac {|x-a|^2 a - (1-|a|^2)(x-a)} {[x,a]^2},
\end{equation}
where
\bea
[x,a]:=(1-2x\cdot a +|x|^2|a|^2)^{1/2}.   \label{1.15b}
\eea
As is well known, $\varphi_a$ is an automorphism (analytic bijection) of ${\bbB_n}$ onto itself, which maps the point $a$ to the origin, and it is also an involution, i.e., $\varphi_a\circ \varphi_a (z) = z$ for all $z \in {\bbB_n}$.

The technical challenge of our paper arises from the decomposition of
the unit ball into spherical boxes and a tricky integration by parts
argument. Accordingly, it is important to introduce suitable  combinatorial
notation. First, we will use  standard multi-index notation so that for a multi-index $\alpha = (\alpha_1, \ldots,
\alpha_n) \in \bbN_0^n$, we denote  $|\alpha| := \alpha_1 + \ldots + \alpha_n$,
and if
$\beta = (\beta_1, \ldots, \beta_n) \in \bbN_0^n$, then $\beta \leq \alpha$
means that $\beta_k \leq \alpha_k$ for all $k$. Given a multi-index
$\alpha$, the corresponding partial differential operator
acting on functions with $n$ real variables   is  defined by
\bea
D^\alpha  = \partial_1^{\alpha_1} \ldots \partial_n^{\alpha_n},\notag
\eea
where $\partial_k = \partial / \partial x_k$ for all $k$. Differentiation will be
performed both in Cartesian and spherical coordinates in the following.

Moreover, in order to perform the integration by parts, we need some notation for certain parts of the boundary. For $x,y \in \bbQ_n$ and multi-indices $\alpha \in \{ 0,1 \}^n$, let
\begin{equation} \label{1.22}
Q_{\alpha}(x,y) = \prod\limits_{\substack{k = 1,\ldots,n \\ \alpha_k = 1}} [x_k,y_k] \times \prod\limits_{\substack{k = 1,\ldots,n \\ \alpha_k = 0}} \{y_k\}.
\end{equation}
For $z \in Q(x,y)$ we denote by $z_{\alpha}$ the element in $Q(x,y)$ with coordinates
\begin{equation} \label{1.50}
(z_{\alpha})_k = \begin{cases} z_k & \text{if } \alpha_k = 1, \\ y_k & \text{if } \alpha_k = 0. \end{cases}
\end{equation}

\section{Preparatory lemmas.}
\label{sec2}

\subsection{Remarks concerning the hyperbolic metric}
\label{sec2.1}
The following facts about M\"obius transforms $\varphi_a$ are well known and easy to check from the definitions, cf. \eqref{1.4}, \eqref{1.15a}, \eqref{1.15b}.

\BEL\label{lem:phi_a}
Let $a\in {\bbB_n}$. The identities
\begin{equation}\label{2.1a}
1-|\varphi_a(x)|^2 = \frac {w(a)w(x)} {[x,a]^2}
\ \ \mbox{and} \ \ \left|\varphi_a^{\prime}(x)\right|=\frac {w(a) }{[x,a]^2}
\end{equation}
hold for every $x\in {\bbB_n}$. Moreover, for all  $x,y\in {\bbB_n}$,
\begin{equation}\label{eqn:composition}
[\varphi_a(x),\varphi_a(y)]^2= \left|\varphi_a^{\prime}(x)\right|\left|\varphi_a^{\prime}(y)\right|[x,y]^2.
\end{equation}
\ENL

The distance of two points
$a,b\in {\bbB_n}$ in the hyperbolic (or Poincar\'e) metric  is given by
\[
d(a,b)= \log \left(\frac {1+|\varphi_a(b)|}{1-|\varphi_a(b)|}\right),
\]
and, in particular,
$
d(0,a)=\log \big( (1+|a|)/(1-|a|)\big) .
$

\BEL\label{lem:controlledbymetric}
Let $a,b\in {\bbB_n}$. Then
\[e^{-d(a,b)} ~\leq~ \frac {[x,a]}{[x,b]} ~\leq~ e^{d(a,b)}\]
for all $x\in {\bbB_n}$.
\ENL

\begin{proof}
We write $x=\varphi_{a} (y)$ and $b=\varphi_{a}(c)$. By \eqref{eqn:composition}, we have
\begin{align*}
\frac {[x,a]^2}{[x,b]^2} = \frac {[\varphi_a (y),\varphi_a (0)]^2}{[\varphi_a (y),\varphi_a (c)]^2}
= \frac {|\varphi_a^{\prime}(y)||\varphi_a^{\prime}(0)|[y,0]^2} {|\varphi_a^{\prime}(y)|
|\varphi_a^{\prime}(c)|[y,c]^2} = \frac {[a,c]^2}{[y,c]^2}
\end{align*}
using \eqref{2.1a} for $|\varphi_a^{\prime}(0)|$ and $|\varphi_a^{\prime}(c)|$. Since $c=\varphi_a(b)$, we have
\[
\frac {[x,a]}{[x,b]} = \frac {[a,c]}{[y,c]} \leq \frac {1+|c|}{1-|c|}= \frac {1+|\varphi_a(b)|}{1-|\varphi_a(b)|}=e^{d(a,b)}.
\]
The other inequality follows from interchanging the roles of $a$ and $b$.
\end{proof}

Next we prove some crucial properties of the dyadic boxes $B_j$ and their enlargements $B_j^*$. Recall that the $B_j$ are essentially disjoint and $B_j^* = B_j + 2^{-m-2}\bbB_n$.

\BEL\label{lem:propertiesofQ}
Let $B_j$ be a dyadic box of generation $m$. Then
\begin{enumerate}
\item[\rm(i)]
$|B_j| \approx 2^{-mn}$,
\item[\rm(ii)]
$\diam(B_j) \approx 2^{-m}$,
\item[\rm(iii)]
$|B_j^*| \approx 2^{-mn}$,
\item[\rm(iv)]
there is a constant $N \in \bbN$ such that every $x \in \bbB_n$ is contained in at most $N$ of the sets $B_j^*$,
\item[\rm(v)]
$1-|x|\approx 1-|x|^2 \approx 2^{-m}$ whenever $x\in B_j^*$,
\item[\rm(vi)]
$[x,a]\approx [x,b]$ for all $x\in {\bbB_n}$ and $a,b\in B_j^*$.
\end{enumerate}
\ENL

\begin{proof}
Without loss of generality, we may always assume that $B_j = \sigma(C_k)$ for some $k \in \cK$.

(i) A direct computation gives
\begin{align*}
|B_j| &= \int_{1-2^{-m}}^{1-2^{-m-1}} \int_{\tfrac{\pi}{2} k_2 2^{-m}}^{\tfrac{\pi}{2} (k_2+1) 2^{-m}} \ldots \int_{\tfrac{\pi}{2}\tfrac{k_{n-1}}{k_{n-2}+1}}^{\tfrac{\pi}{2}\tfrac{k_{n-1}+1}{k_{n-2}+1}} \int_{2\pi\tfrac{k_n}{k_{n-1}+1}}^{2\pi\tfrac{k_n+1}{k_{n-1}+1}}\\
&\qquad \qquad \qquad \qquad \qquad \qquad r^{n-1} \sin^{n-2}\theta_2 \cdots \sin\theta_{n-1} \, \mathrm{d}\theta_n \cdots \, \mathrm{d}\theta_2 \, \mathrm{d}r\\
&\approx \int_{1-2^{-m}}^{1-2^{-m-1}} \int_{\tfrac{\pi}{2} k_2 2^{-m}}^{\tfrac{\pi}{2} (k_2+1) 2^{-m}} \ldots \int_{\tfrac{\pi}{2}\tfrac{k_{n-1}}{k_{n-2}+1}}^{\tfrac{\pi}{2}\tfrac{k_{n-1}+1}{k_{n-2}+1}} \int_{2\pi\tfrac{k_n}{k_{n-1}+1}}^{2\pi\tfrac{k_n+1}{k_{n-1}+1}}\\
&\qquad \qquad \qquad \qquad \qquad \qquad r^{n-1}\theta_2^{n-2} \cdots \theta_{n-1} \, \mathrm{d}\theta_n \cdots \, \mathrm{d}\theta_2 \, \mathrm{d}r\\
&\approx \left((1-2^{-m-1})^n - (1-2^{-m})^n\right)\left(((k_2+1) 2^{-m})^{n-1} - (k_2 2^{-m})^{n-1}\right)\\
&\qquad \times \prod_{l = 3}^n \left(\left(\frac{k_l+1}{k_{l-1}+1}\right)^{n-l+1} - \left(\frac{k_l}{k_{l-1}+1}\right)^{n-l+1}\right)\\
&\approx 2^{-m}\left((k_2+1)^{n-1} - k_2^{n-1}\right)2^{-m(n-1)}\\
&\qquad \times \prod_{l = 3}^n \frac{1}{(k_{l-1}+1)^{n-l+1}}\left((k_l+1)^{n-l+1} - k_l^{n-l+1}\right)\\
&\approx 2^{-mn} \prod_{l = 2}^{n-1} \frac{(k_l+1)^{n-l+1} - k_l^{n-l+1}}{(k_l+1)^{n-l}}\\
&\approx 2^{-mn}.
\end{align*}

(ii) The edges of the dyadic boxes are coordinate lines $\gamma_l$ of the coordinate transform $\sigma$. Those are of the form
\begin{align*}
\gamma_l(t) &= \sigma(r,\theta_2,\ldots,t,\ldots,\theta_n).\\
&\qquad \qquad \qquad\; \overset{\uparrow}\theta_l-\text{coordinate}
\end{align*}
The length $\ell$ of the edges can therefore be computed directly:
\begin{align*}
\ell(\gamma_1) &= \int_{1-2^{-m}}^{1-2^{-m-1}} \|\gamma_1'(t)\| \, \mathrm{d}t = \int_{1-2^{-m}}^{1-2^{-m-1}} 1 \, \mathrm{d}t = 2^{-m-1},\\
\ell(\gamma_2) &= \int_{\tfrac{\pi}{2} k_2 2^{-m}}^{\tfrac{\pi}{2} (k_2+1) 2^{-m}} \|\gamma_2'(t)\| \, \mathrm{d}t = \int_{\tfrac{\pi}{2} k_2 2^{-m}}^{\tfrac{\pi}{2} (k_2+1) 2^{-m}} r \, \mathrm{d}t \lesssim 2^{-m},\\
&\vdots\\
\ell(\gamma_l) &= \int_{\tfrac{\pi}{2}\tfrac{k_l}{k_{l-1}+1}}^{\tfrac{\pi}{2}\tfrac{k_l+1}{k_{l-1}+1}} \|\gamma_l'(t)\| \, \mathrm{d}t = \int_{\tfrac{\pi}{2}\tfrac{k_l}{k_{l-1}+1}}^{\tfrac{\pi}{2}\tfrac{k_l+1}{k_{l-1}+1}} r\sin\theta_2\cdots \sin\theta_{l-1} \, \mathrm{d}t\\
&\lesssim (k_2+1) 2^{-m} \cdot \frac{k_3+1}{k_2+1} \cdot \ldots \cdot \frac{k_{l-1}+1}{k_{l-2}+1} \cdot \frac{1}{k_{l-1}+1} = 2^{-m},\\
&\vdots\\
\ell(\gamma_n) &= \int_{2\pi\tfrac{k_n}{k_{n-1}+1}}^{\tfrac{\pi}{2}\tfrac{k_n+1}{k_{n-1}+1}} \|\gamma_n'(t)\| \, \mathrm{d}t = \int_{\tfrac{\pi}{2}\tfrac{k_n}{k_{n-1}+1}}^{\tfrac{\pi}{2}\tfrac{k_n+1}{k_{n-1}+1}} r\sin\theta_2\cdots \sin\theta_{n-1} \, \mathrm{d}t\\
&\lesssim (k_2+1) 2^{-m} \cdot \frac{k_3+1}{k_2+1} \cdot \ldots \cdot \frac{k_{n-1}+1}{k_{n-2}+1} \cdot \frac{1}{k_{n-1}+1} = 2^{-m}.
\end{align*}
It follows $\diam(B_j) \approx 2^{-m}$.

(iii) We only need to show $|B_j^*| \lesssim 2^{-mn}$. Clearly, $B_j^*$ is contained in a ball of diameter at most $\diam(B_j) + 2^{-m-1} \approx 2^{-m}$. Hence $|B_j^*| \lesssim 2^{-mn}$.

(iv) Let $x \in \bbB_n$ with $|x| \in [1-2^{-m},1-2^{-m-1}]$ and denote the set of indices $j$ such that $B_j \cap E(x,2^{-m-2}) \neq \emptyset$ by $\cJ_x$. It suffices to show that the cardinality of $\cJ_x$ is bounded independently of $x$. Clearly, $B_j$ is of generation $m-1$, $m$ or $m+1$ for every $j \in \cJ_x$. Therefore, by (i) and (ii), $|B_j| \approx 2^{-nm}$ and $\diam(B_j) \approx 2^{-m}$ for every $j \in \cJ_x$. It follows that there is a constant $C$ (only depending on $n$) such that $B_j \subset E(x,C2^{-m})$ for every $j \in \cJ_x$. Since $|E(x,C2^{-m})| \approx 2^{-mn}$, $E(x,C2^{-m})$ can only contain finitely many essentially disjoint sets of volume $\approx 2^{-mn}$. This finite number is the constant $N$ we are looking for; it only depends on $n$. This finishes the proof.

(v) As $1 - 2^{-m} \leq |x| \leq 1 - 2^{-m-1}$ for $x$ in a dyadic box of generation $m$, this is clear.

(vi) In view of (ii), there is a constant $C $ such that $|a-b|\leq C 2^{-m}$
for all $a,b\in B_j^{\ast}$. We also have $1-|a|^2\geq 2^{-m-2}$ and $1-|b|^2\geq 2^{-m-2}$. Hence, by \eqref{1.15b} and \eqref{2.1a},
\begin{align*}
	\frac {1}{1-|\varphi_a(b)|^2} &= \frac{[b,a]^2}{w(a)w(b)}
	=\frac{w(a)w(b) + |a|^2 + |b| ^2 - 2 a\cdot b }{w(a)w(b)} \\
	&=1+\frac {|a-b|^2} {w(a)w(b)} \leq 1+16C^2,
\end{align*}
and therefore,
\[
d(a,b) = \log \left(\frac {(1+|\varphi_a(b)|)^2}{1-|\varphi_a(b)|^2}\right) \leq \log\left(\frac {4} {1-|\varphi_a(b)|^2}\right)\leq \log(4+64C^2).
\]
The assertion (vi) now follows from Lemma \ref{lem:controlledbymetric}.
\end{proof}

\subsection{Basic facts about the space \texorpdfstring{$b_{\lambda}^p$}{bp}}
\label{sec2.2}
We will need a number of results on the spaces $b_{\lambda}^p$. The
following was proven in {\cite[Lemma 2.8]{JP}}.

\BEL
\label{lem:derivofkernel}
For all multi-indices $\alpha \in \bbN_0^n$ and all $\lambda > -1$,
\[
|D_x^{\alpha} R_{\lambda}(x,y)| \lesssim [x,y]^{-n-\lambda-|\alpha|} \ , \ x,y \in {\bbB_n}.
\]
\ENL

To prove the next lemma, we use the Forelli--Rudin type estimate
\begin{equation}\label{e:liu}
	\int_{\bbB_n} \frac{(1-|y|^2)^t}{[x,y]^{n+s+t}} \, \mathrm{d}V(y)\approx (1-|x|^2)^{-s},
\end{equation}
where $s>0$ and $t>-1$ (see \cite[Proposition 2.2]{LS}).

\BEL\label{lem:intgop}
Let $\lambda>-1$ and $1<p<\infty$. Then the integral operator defined by
\begin{equation*}
\Lambda_{\lambda} f(x):=\int\limits_{\bbB_n} \frac {f(y)}{[x,y]^{n+\lambda}} \, \mathrm{d}V_{\lambda}(y), \quad x\in {\bbB_n},
\end{equation*}
is bounded on $L_{\lambda}^p$.
\ENL
\begin{proof}
Define $h(x) = (1-|x|^2)^\alpha$ and choose $\frac{-1-\lambda}{\max\{p,q\}}<\alpha<0$. Then, by \eqref{e:liu}, when $1/p+1/q=1$, we get
$$
	\int_{\bbB_n} \frac{h(y)^q}{[x,y]^{n+\lambda}} \, \mathrm{d}V_\lambda(y)
	= \int_{\bbB_n} \frac{(1-|y|^2)^{\alpha q + \lambda}}{[x,y]^{n+\lambda}} \, \mathrm{d}V(y)
	\lesssim h(x)^q
$$
and similarly $\int_{\bbB_n} \frac{h(x)^p}{[x,y]^{n+\lambda}} \, \mathrm{d}V_\lambda(x) \lesssim h(y)^p$. Now Schur's test completes the proof.
\end{proof}

\BEL\label{lem:normdom}
Let $\lambda>-1$, $1<p<\infty$ and $k\in \mathbb{N}$. If $\alpha$ is a multi-index with $|\alpha| = k$ and $f\in b_{\lambda}^p$, then
$w^k D^\alpha f\in L_{\lambda}^p$, and
\[
\left\|w^k D^\alpha f\right\|_{p,\lambda} ~\lesssim~ \|f\|_{p,\lambda}.
\]
\ENL

\begin{proof}
In view of Lemma \ref{lem:pointwiseest}, $b_{\lambda}^p \subset b_{\gamma}^2$
for sufficiently large $\gamma$, and hence
\[
f(x)=\int\limits_{\bbB_n} R_{\gamma}(x,y)f(y) \, \mathrm{d}V_{\gamma}(y).
\]
Differentiating under the integral sign, we obtain
\begin{equation}
\label{eqn:reprnofderiv}
w(x)^{|\alpha|} D^{\alpha} f(x)=\int\limits_{\bbB_n} K(x,y) f(y) \, \mathrm{d}V_{\lambda}(y),
\end{equation}
where
\[K(x,y):=w(x)^{|\alpha|} w(y)^{\gamma-\lambda} D_x^{\alpha}R_{\gamma}(x,y).\]
By Lemma \ref{lem:derivofkernel},
\[|K(x,y)| \lesssim \frac {w(x)^{|\alpha|}w(y)^{\gamma-\lambda}} {[x,y]^{n+\gamma+|\alpha|}} \lesssim \frac {1} {[x,y]^{n+\lambda}},\]
where the last inequality follows from
\[[x,y] \geq 1-|x||y| \geq \max\{1-|x|,1-|y|\} \geq \frac{1}{2}\max\{w(x),w(y)\}.\]
Using \eqref{eqn:reprnofderiv} and Lemma \ref{lem:intgop}, we get
\[\left\|w^{|\alpha|} D^{\alpha} f\right\|_{p,\lambda} \lesssim \|f\|_{p,\lambda},\]
which completes the proof.
\end{proof}

\subsection{Integration by parts in spherical boxes.}
\label{sec2.3}

The results in \cite{TV2} were based on a tricky integration by parts. We have to
exploit this procedure in even higher generality, hence, it is useful to
expose the corresponding general integration-by-parts-formula. The set of $n$-times continuously
differentiable functions on $\bbB_n$ will be denoted by $C^n(\bbB_n)$. In the following, the Jacobian determinant of the coordinate transform $\sigma$ will be denoted by $J_{\sigma}$.

\BEL
\label{lem299}
Let $f \in L^1_{\rm loc}$, $g \in C^n(\bbB_n)$ and $x,y \in \bbQ_n$. Then, with $F= f \circ \sigma$ and $ G= g \circ\sigma$,
\begin{align} \label{2.33}
\int\limits_{B(x,y)}  fg \, \mathrm{d}V_{\lambda} &= \int\limits_{Q(x,y)} F(\gamma) G(\gamma) J_{\sigma}(\gamma) w (\sigma(\gamma))^\lambda \, \mathrm{d}\gamma\notag\\
&= \sum_{\alpha\in \{0,1\}^n} s_{\alpha} \int\limits_{Q_\alpha(x,y)} \bigg(\int\limits_{Q(x,\gamma_\alpha )} F(\tau) J_{\sigma}(\tau) w(\sigma(\tau))^\lambda \, \mathrm{d}\tau  \bigg)\\
&\qquad \qquad \qquad \qquad \qquad \times D^\alpha G(\gamma_\alpha ) \, \mathrm{d} \gamma_\alpha,\notag
\end{align}
where $s_{\alpha} \in \{\pm 1\}$ for all $\alpha \in \{0,1\}^n$.
\ENL

Recall that $Q_{\alpha}(x,y) = \prod\limits_{\substack{k = 1,\ldots,n \\ \alpha_k = 1}} [x_k,y_k] \times \prod\limits_{\substack{k = 1,\ldots,n \\ \alpha_k = 0}} \{y_k\}$, which means that the integration  $\int_{Q_{\alpha}(x,y)} \, \mathrm{d}\gamma_\alpha$ is performed only in those variables $\gamma_k$ where $\alpha_k = 1$. This is in concordance with the notation \eqref{1.50} since $(\gamma_\alpha)_k = y_k$ if $\alpha_k = 0$.

\begin{proof}
Let $\mathbbm{1} := (1,\ldots,1)$. We will use the following well-known formula
\begin{equation} \label{eqn:product_rule}
(D^{\mathbbm{1}}u) \cdot v = \sum\limits_{\alpha \in \{0,1\}^n} (-1)^{|\alpha|} D^{\mathbbm{1}-\alpha}(u \cdot D^{\alpha}v),
\end{equation}
which is easily proven by induction. Now choose
\[u(\gamma) := \int_{Q(x,\gamma)} F(\tau)J_{\sigma}(\tau)w(\sigma(\tau))^{\lambda} \, \mathrm{d}\tau, \quad \text{and} \quad v(\gamma) := G(\gamma),\]
and observe that $u(\gamma) = 0$ if $x_k = \gamma_k$ for some $k$. Integrating the formula \eqref{eqn:product_rule} over $Q(x,y)$ and using the fundamental theorem of calculus multiple times yields the result.
\end{proof}

\section{Proofs of the main results.}
\label{sec3}

Before giving the proofs of our main results, we still need to consider four lemmas. The first one is used to fix a small flaw in
\cite{TV2}: in the reference, the inequality (3.8) is not true as such, but the
integration domain has to be replaced by a larger set. This, however, is not
difficult, and we use here the enlarged dyadic boxes $B_j^*$ to this end.

\BEL\label{lem:mvp}
Let $f$ be a harmonic function on ${\bbB_n}$ and let $B_j$ be one of our dyadic boxes. Then for each $j$,
\begin{equation}\label{eqn:meanvalueineq}
|f(x)|\lesssim \frac {1}{|{B_j}|_{\lambda}} \int\limits_{{B_j^*}} |f| \, \mathrm{d}V_{\lambda}
\end{equation}
for every $x\in {B_j}$.
\ENL

\begin{proof}
Suppose that $B_j$ is of generation $m$. By definition of the sets $B_j^*$, we have $E(x,\delta 2^{-m-2})\subset {B_j^*}$ for every $x\in {B_j}$. The mean value property of harmonic functions
yields
\[f(x)=\frac {1}{|E(x,2^{-m-2})|} \int\limits_{E(x,2^{-m-2})} f(y) \, \mathrm{d}V(y).\]
Thus,
\begin{align*}
|f(x)| &\leq \frac {1}{|E(x,2^{-m-2})|} \int\limits_{E(x,2^{-m-2})} |f(y)| \, \mathrm{d}V(y)\\
&\lesssim 2^{nm} \int\limits_{{B_j^*}} |f(y)| \, \mathrm{d}V(y)\\
&\lesssim 2^{(n+\lambda)m} \int\limits_{{B_j^*}} |f(y)| w(y)^{\lambda} \, \mathrm{d}V(y),
\end{align*}
where the last inequality follows from the fact that $w(y) \approx 2^{-m}$ for $y\in {B_j^*}$; see Lemma \ref{lem:propertiesofQ}. Since $|{B_j}|_{\lambda}\approx 2^{-(n+\lambda)m}$,
this proves \eqref{eqn:meanvalueineq}.
\end{proof}

In order to apply Lemma \ref{lem299}, we need to transfer partial derivatives from spherical coordinates to Cartesian coordinates. For this we need the following lemma. We note that there is an explicit formula for iterated partial derivatives, the Fa\`a di Bruno formula, but since we do not need it in full generality, we decided to give a proof in simple terms.

\begin{lemma} \label{lem:iterated_derivatives}
Let $g \in C^n(\bbB_n)$ and $\alpha \in \{0,1\}^n$, $|\alpha| \geq 1$. Then for every $\beta \in \bbN_0^n$ with $1 \leq |\beta| \leq |\alpha|$ there is a $d_{\beta} \in C^{\infty}(\bbQ_n)$ such that
\[D^{\alpha}_{\gamma} g(\sigma(\gamma)) = \sum\limits_{1 \leq |\beta| \leq |\alpha|} c_{\beta}(\gamma) d_{\beta}(\gamma)(D^{\beta}_x g)(x),\]
where
\[c_{\beta}(\gamma) = \prod\limits_{j = 2}^{n-1} (\sin\theta_j)^{\max\big\{0,\, |\beta| - \sum\limits_{i = 1}^j \alpha_i\big\}},\]
$\gamma = (r,\theta_2,\ldots,\theta_n) \in \bbQ_n$ and $x = \sigma(\gamma)$.
\end{lemma}

\begin{proof}
We only prove that one can factor out $(\sin\theta_2)^{|\beta|-\alpha_1-\alpha_2}$. As will be clear from the proof, factoring out the other terms in the product can then be done in the same way. We also note that the power of each $\sin\theta_j$ in the expansion of $D^{\alpha}_{\gamma}(g \circ \sigma)$ can never get negative. It can therefore be assumed that $|\beta| - \alpha_1 - \alpha_2 \geq 1$; otherwise there is nothing to prove.

We prove the claim by induction over $|\alpha|$. First assume that $|\alpha| = 1$. Then
\[D^{\alpha}_{\gamma} (g \circ \sigma) = \frac{\partial (g \circ \sigma)}{\partial \gamma_k} = \sum\limits_{l = 1}^n \frac{\partial x_l}{\partial \gamma_k} \frac{\partial g}{\partial x_l}\]
for some $k \in \{1,\ldots,n\}$. Now observe that $\sin\theta_2$ is necessarily a factor of $\frac{\partial x_l}{\partial \gamma_k}$ unless $k \in \{1,2\}$. This completes the proof for $|\alpha| = 1$.

To complete the induction process, we compute
\[\frac{\partial}{\partial \gamma_k} \sum\limits_{1 \leq |\beta| \leq |\alpha|} (\sin\theta_2)^{|\beta|-\alpha_1-\alpha_2} d_{\beta}D^{\beta}_x g\]
for each $k \in \{1,\ldots,n\}$ with $\alpha_k = 0$ and show that we can factor out sufficiently many $\sin\theta_2$. First assume $k \neq 2$. Then
\begin{align*}
&\frac{\partial}{\partial \gamma_k} \sum\limits_{1 \leq |\beta| \leq |\alpha|} (\sin\theta_2)^{|\beta|-\alpha_1-\alpha_2} d_{\beta}D^{\beta}_x g\\
&= \sum\limits_{1 \leq |\beta| \leq |\alpha|} (\sin\theta_2)^{|\beta|-\alpha_1-\alpha_2} \left(\frac{\partial d_{\beta}}{\partial \gamma_k}D^{\beta}_x g + d_{\beta}\sum\limits_{l = 1}^n \frac{\partial x_l}{\partial \gamma_k} \frac{\partial}{\partial x_l} D^{\beta}_x g\right).
\end{align*}
Let $\alpha'$ be equal to $\alpha$ except that the $k$-th entry is flipped from $0$ to $1$. If $k \neq 1$, we can factor out another $\sin\theta_2$ from $\frac{\partial x_l}{\partial \gamma_k}$ as above and therefore the power of $\sin\theta_2$ again matches the order of the derivative in each term. If $k = 1$, then $\frac{\partial x_1}{\partial \gamma_1} = \cos\theta_2$, so there is no additional $\sin\theta_2$ to factor out. However, since $\alpha'_1 = \alpha_1 + 1$, we still have the correct power of $\sin\theta_2$ in the expansion. Now let $k = 2$. In this case
\begin{align*}
&\frac{\partial}{\partial \gamma_k} \sum\limits_{1 \leq |\beta| \leq |\alpha|} (\sin\theta_2)^{|\beta|-\alpha_1-\alpha_2} d_{\beta}D^{\beta}_x g\\
&= \sum\limits_{1 \leq |\beta| \leq |\alpha|} (\sin\theta_2)^{|\beta|-\alpha_1-\alpha_2} \left(\frac{\partial d_{\beta}}{\partial \gamma_k}D^{\beta}_x g + d_{\beta}\sum\limits_{l = 1}^n \frac{\partial x_l}{\partial \gamma_k} \frac{\partial}{\partial x_l} D^{\beta}_x g\right)\\
& \qquad \qquad \quad  + (\sin\theta_2)^{|\beta|-\alpha_1-\alpha_2-1} \cos\theta_2 d_{\beta}D^{\beta}_x g.
\end{align*}
As $\alpha'_2 = \alpha_2 + 1$, we are in the same situation as for $k = 1$. We again have the correct power of $\sin\theta_2$ in each term. This completes the proof. 
\end{proof}

The functions $c_{\beta}$ have the following important property.

\begin{lemma} \label{lem:5.3}
Let $\alpha \in \{0,1\}^n$, $|\alpha| \geq 1$ and $c_{\beta}(\gamma)$ be as above. Then
\[\int_{Q_{\alpha}(x^{(j)},y^{(j)})} c_{\beta}(\gamma) \, \mathrm{d}\gamma_{\alpha} \lesssim 2^{-m|\beta|}\]
for all $j \in \mathbb{N}$.
\end{lemma}

\begin{proof}
As usual, we may assume that $Q(x^{(j)},y^{(j)}) = C_k$ for some $k \in \cK$. Choose $h \in \mathbb{N}$ such that $|\beta| = \sum\limits_{i = 1}^h \alpha_i$. If $h = 1$, then $|\beta| = \alpha_1 = 1$ as well as $c_{\beta}(\gamma) = 1$. Therefore,
\[\int_{Q_{\alpha}(x^{(j)},y^{(j)})} c_{\beta}(\gamma) \, \mathrm{d}\gamma_{\alpha} \lesssim 2^{-m} = 2^{-m|\beta|}.\]
If $h \geq 2$, another direct computation shows
\begin{align*}
\int_{Q_{\alpha}(x^{(j)},y^{(j)})} c_{\beta}(\gamma) \, \mathrm{d}\gamma_{\alpha} &\lesssim 2^{-m\alpha_1} (k_2+1)^{|\beta|-\alpha_1-\alpha_2}2^{-m(|\beta|-\alpha_1)}\\
&\qquad \times \frac{(k_3+1)^{|\beta|-\alpha_1-\alpha_2-\alpha_3}}{(k_2+1)^{|\beta|-\alpha_1-\alpha_2}} \cdots \frac{(k_{n-1}+1)^{|\beta|-\sum\limits_{l = 1}^h \alpha_l}}{(k_{n-2}+1)^{|\beta|-\sum\limits_{l = 1}^{h-1} \alpha_l}}\\
&= 2^{-m|\beta|}.\qedhere
\end{align*}
\end{proof}

The following lemma is the most important technical step in the proof of our theorems.

\BEL
\label{lem3.2v}
Let $\psi \in L_ {\rm loc}^1$. For $z \in {\bbB_n}$ and every dyadic box ${B_j}$, we have
\[\left|T_{\psi} (\chi_{{B_j}} f)(z)\right|
\lesssim \frac {\widehat{\psi}_j}{|B_j|_{\lambda}} \sum_{|\beta| \leq n} \int\limits_{{B_j^*}}
\frac { w(y)^{|\beta|} |D^\beta f(y)|}{[z,y]^{n+\lambda}} \, \mathrm{d}V_{\lambda}(y).\]
\ENL

\begin{proof}
Suppose that the generation of ${B_j}$ is $m$. We may further assume that $B_j = \sigma(C_k)$ for some $k \in \cK$. Let $\gamma = (r,\theta_2,\ldots,\theta_n) \in \bbQ_n$ be such that $y = \sigma(\gamma) \in B_j$ and $\alpha \in \{0,1\}^n$, $|\alpha| \geq 1$. By Lemma \ref{lem:derivofkernel}, we have
\begin{equation}
|D_y^{\alpha}R_{\lambda}(z,y)| \lesssim [z,y]^{-n-\lambda-|\alpha|} \label{3.1c}
\end{equation}
for every $z \in \bbB_n$. Using Lemma \ref{lem:iterated_derivatives}, we obtain
\begin{align*}
\left| D_\gamma^{\alpha}\big( R_\lambda(z, \sigma(\gamma))f( \sigma(\gamma)) \big)\right| &\leq \sum\limits_{1 \leq |\beta| \leq |\alpha|} c_{\beta}(\gamma) |d_{\beta}(\gamma)|\left|D^{\beta}_y \big(R_\lambda(z,y)f(y)\big)\right|\\
&\lesssim \sum\limits_{1 \leq |\beta| \leq |\alpha|} c_{\beta}(\gamma) \sum\limits_{\tilde{\beta} \leq \beta} \big|D^{\beta-\tilde{\beta}}_y R_\lambda(z,y)\big| \big|D^{\tilde{\beta}}_y f(y)\big|\\
&\lesssim \sum\limits_{1 \leq |\beta| \leq |\alpha|} c_{\beta}(\gamma) \sum\limits_{\tilde{\beta} \leq \beta} 2^{m(|\beta| - |\tilde{\beta}|)} [z,y]^{-n-\lambda}\\
&\qquad \qquad \qquad \qquad \quad \times \big|D^{\tilde{\beta}}_y f(y)\big|
\end{align*}
where the last inequality follows from \eqref{3.1c} and $[z,\sigma(\gamma)] \geq 1 - |\sigma(\gamma)| \approx 2^{-m}$ by Lemma \ref{lem:propertiesofQ}. 
We now apply Lemma \ref{lem:mvp} to the functions
$D^{\tilde{\beta}} f$ and use (v) and (vi) of Lemma \ref{lem:propertiesofQ}.
This yields
\begin{align} \label{3.10}
&\left| D_\gamma^{\alpha}\big( R_\lambda(z,\sigma(\gamma))f(\sigma(\gamma)) \big)\right|\notag\\
&\lesssim \sum\limits_{1 \leq |\beta| \leq |\alpha|} c_{\beta}(\gamma) \sum\limits_{\tilde{\beta} \leq \beta} 2^{m(|\beta| - |\tilde{\beta}|)} [z,y]^{-n-\lambda} \frac {1}{|{B_j}|_{\lambda}}\int\limits_{{B_j^*}} |D^{\tilde{\beta}} f | \, \mathrm{d}V_{\lambda}\notag\\
&\lesssim \sum\limits_{1 \leq |\beta| \leq |\alpha|} c_{\beta}(\gamma) \sum\limits_{\tilde{\beta} \leq \beta} 2^{m|\beta|} \frac {1}{|{B_j}|_{\lambda}}\int\limits_{{B_j^*}} \frac { w(y)^{|\tilde{\beta}|} |D^{\tilde{\beta}} f(y)|}{[z,y]^{n+\lambda}} \, \mathrm{d}V_{\lambda}(y)\\
&\lesssim \sum\limits_{1 \leq |\beta| \leq |\alpha|} c_{\beta}(\gamma) 2^{m|\beta|} \frac {1}{|{B_j}|_{\lambda}}\int\limits_{{B_j^*}} \frac { w(y)^{|\beta|} |D^{\beta} f(y)|}{[z,y]^{n+\lambda}} \, \mathrm{d}V_{\lambda}(y)\notag
\end{align}
for all $z \in \bbB_n$. We are now going to apply the integration by parts lemma, Lemma \ref{lem299}. Recall that $B_j = B(x^{(j)},y^{(j)}) = \sigma(Q(x^{(j)},y^{(j)}))$. Integrating \eqref{3.10} over $Q_{\alpha}(x^{(j)},y^{(j)})$ and using Lemma \ref{lem:5.3} yields
\begin{align} \label{5.4}
&\int_{Q_{\alpha}(x^{(j)},y^{(j)})} \left| D_\gamma^{\alpha}\big( R_\lambda(z, \sigma(\gamma))f( \sigma(\gamma)) \big)\right| \, \mathrm{d}\gamma_{\alpha}\notag\\
&\qquad \qquad \qquad \qquad \qquad \quad \lesssim \sum_{1 \leq |\beta| \leq |\alpha| } \frac {1}{|{B_j}|_{\lambda}} \int\limits_{{B_j^*}} \frac { w(y)^{|\beta|} |D^{\beta} f(y)|}{[z,y]^{n+\lambda}} \, \mathrm{d}V_{\lambda}(y).
\end{align}
We now  apply \eqref{2.33} to the integral
\bea
T_\psi (\chi_{B_j} f)  (z) = \int\limits_{B_j} \psi(y) f(y) R_\lambda(z,y) \, \mathrm{d}V_\lambda (y).
\label{3.1p}
\eea
For the factor $f$  in \eqref{2.33} we take the function $\psi $ and for $g$ the
function  $R_{\lambda}(z, \cdot) f $ with a fixed $z \in \bbB_n$.
This yields (see the remark just after this proof)
\begin{align} \label{3.27}
&\bigg|\int\limits_{{B_j}} \psi(y)  f(y)R_{\lambda}(z,y) \, \mathrm{d}V_{\lambda}(y)\bigg|\notag\\
&\leq \sum_{\alpha\in \{0,1\}^n}
\int\limits_{Q_\alpha(x^{(j)},y^{(j)})} \bigg| \int\limits_{Q(x^{(j)},\gamma_\alpha)} \psi ( \sigma(\tau)) w(\sigma(\tau))^\lambda J_{\sigma}(\tau) \, \mathrm{d}\tau \bigg|\\
& \qquad \qquad \qquad \qquad \times \Big|
D_\gamma^\alpha  \big( R_\lambda(z, \sigma(\gamma_\alpha)) f(\sigma (\gamma_\alpha )) \big) \Big| \, \mathrm{d}\gamma_\alpha. \notag
\end{align}

The change of the variables formula turns the integral over the set $Q(x^{(j)}, \gamma_\alpha)$ into
\bea
\bigg|\int\limits_{B(x^{(j)}, y)} \psi \, \mathrm{d}V_\lambda\bigg| \leq \widehat{\psi}_j
\label{3.27a}
\eea

for $y = \sigma(\gamma_\alpha) \in B_j$, see \eqref{eqn:setfunction}. Applying \eqref{5.4}, we thus see that  \eqref{3.27} is bounded by a
constant times
\bea
\frac{\widehat{\psi}_j}{|B_j|_{\lambda}}\sum_{|\beta|\leq n}
\int\limits_{{B_j^*}} \frac { w(y)^{|\beta|} |D^\beta f(y)|}{[z,y]^{n+\lambda}} \, \mathrm{d}V_{\lambda}(y) ,
\label{3.27z}
\eea
since $ |\alpha| \leq n$.
So the lemma follows by combining \eqref{3.1p}--\eqref{3.27z}.
\end{proof}

We emphasize that the deduction \eqref{3.27} is in the core of our result:
By using integration by parts, it is possible to make estimates only in terms
of the modulus of the integral of $\psi$, and a direct bound involving the modulus of $\psi$ can be avoided.

We are now in the position to prove our main results:

\begin{proof}[Proof of Theorem \ref{thm:mainthm}]
Recall that the Toeplitz operator $T_{\psi}$ is defined with the help of the series
\begin{equation*}
T_\psi f(x)  = \sum_{j=1}^\infty  T_\psi(\chi_{B_j}f)(x),
\end{equation*}
cf.~\eqref{eqn:generalizedToeplitz}. We show that the series converges absolutely for almost every
$x \in \bbB_n$ and that the resulting operator is bounded. Indeed, by the previous lemma, assumption \eqref{eqn:weakCarleson} and Lemma \ref{lem:propertiesofQ}, we have
\begin{align} \label{eqn:absconvergence}
\sum_{j = 1}^\infty \left|T_{\psi} (\chi_{B_{j}} f)(x)\right| &\lesssim C_{\psi} \sum_{j = 1}^\infty \sum_{|\beta| \leq n} \int\limits_{B_{j}^{\ast}}
\frac { w(y)^{|\beta|} |D^\beta f(y)|}{[x,y]^{n+\lambda}} \, \mathrm{d}V_{\lambda}(y) \nonumber \\
&\lesssim {C_{\psi}} \sum_{|\beta| \leq n} \, \int\limits_{\bbB_ n}
\frac { w(y)^{|\beta|} |D^\beta f(y)|}{[x,y]^{n+\lambda}} \, \mathrm{d}V_{\lambda}(y)
\end{align}
for all $x \in \bbB_n$. By Lemmas \ref{lem:intgop} and \ref{lem:normdom}, we see that for all $|\beta|\leq n$,
\[\bigg\|\int\limits_{\bbB_n} \frac{w(y)^{|\beta|}|D^\beta f(y)|}{[x,y]^{n+\lambda}} \, \mathrm{d}V_{\lambda}(y)\bigg\|_{p,\lambda} \lesssim \left\|w^{|\beta|} D^\beta f\right\|_{p,\lambda} \lesssim \|f\|_{p,\lambda}.\]
This implies that the series $\sum_{j=0}^{\infty} \left|T_{\psi} (\chi_{B_{j}} f)(x)\right|$
is pointwise bounded by an $L_{\lambda}^p$-function, and thus it converges for
almost all $x\in {\bbB_n}$. Moreover, the above argument implies
\begin{equation} \label{eqn:Lpbounded}
\bigg\|\sum_{j=1}^{\infty} T_{\psi} (\chi_{B_{j}} f)\bigg\|_{p,\lambda}\lesssim C_{\psi} \|f\|_{p,\lambda}.
\end{equation}
By dominated convergence, the series also converges in $L_\lambda^p$. In particular, $T_\psi f \in b_\lambda^p$ for all $f \in b_\lambda^p$ and $\|T_{\psi}\|\lesssim C_{\psi}$.
\end{proof}

\begin{proof}[Proof of Theorem \ref{limit-thm}]
We first show that $T_{\psi_\rho}f \to T_\psi f$ for every $f \in b_\lambda^p$. As $(\widehat{\psi}_{\rho})_j \leq \widehat{\psi}_j$, the estimate \eqref{eqn:absconvergence} is uniform in $\rho$. That is, for almost every $x \in \bbB_n$, the series $\sum\limits_{j \in \bbN } T_{\psi_\rho} (\chi_{B_{j}} f)(x)$ converges absolutely and uniformly in $\rho$. In particular,
\[\lim\limits_{\rho \to 1} \sum_{j = 1}^\infty T_{\psi_\rho} (\chi_{B_{j}} f)(x) = \sum_{j = 1}^\infty \lim\limits_{\rho \to 1} T_{\psi_\rho} (\chi_{B_{j}} f)(x) = \sum_{j = 1}^\infty T_\psi (\chi_{B_{j}} f)(x).\]
Moreover, $\|T_{\psi_{\rho}}f\|_{p,\lambda} \lesssim C_\psi\|f\|$ by \eqref{eqn:Lpbounded}. By dominated convergence, we get $T_{\psi_\rho}f \to T_\psi f$ as $\rho \to 1$. For the adjoint observe
\[\langle T_\psi f,g\rangle = \sum_{j=1}^\infty  \langle \psi\chi_{B_j}f, g \rangle = \sum_{j=1}^\infty  \langle f, \overline{\psi}\chi_{B_j}g \rangle = \langle f,T_{\overline{\psi}} g \rangle\]
and thus $T_{\overline{\psi}_\rho}g \to T_\psi^* g$ follows analogously.
\end{proof}

\begin{proof}[Proof of Theorem \ref{thm:mainthm2}]
A routine normal families argument shows that the unit ball of $b_{\lambda}^p$ is compact in the topology of uniform convergence on compact subsets of ${\bbB_n}$. It therefore suffices to show that $\|T_{\psi}f_k\|_{p,\lambda}\to 0$ for all norm bounded sequences $( f_k )_{k=1}^{\infty}
\subset b_{\lambda}^p$ which converge to zero uniformly on compacta of ${\bbB_n}$.

So, we fix such a sequence $( f_k )_{k=1}^{\infty}$
with $\Vert f_k \Vert_{p,\lambda} \leq 1$ for all $k$. Let $\varepsilon > 0$ be arbitary. For all $j,k \in \mathbb{N}$ Lemma \ref{lem3.2v} implies
\[\left|T_{\psi} (\chi_{B_j} f_k)(x)\right| \lesssim \frac{\widehat{\psi}_j}{|B_j|_{\lambda}}
\sum_{|\beta| \leq n} \int\limits_{B_j^{\ast}} \frac { w(y)^{|\beta|} |D^\beta  f_k(y)|}{
[x,y]^{n+\lambda}} \, \mathrm{d}V_{\lambda}(y),\]
and the weak vanishing Carleson condition \eqref{eqn:weakvanishingCarleson} implies
\[
\lim_{j\to \infty} \frac{\widehat{\psi}_j}{|B_j|_{\lambda}}=0.
\]
We choose $N\in \mathbb{N}$ such that $\frac{\widehat{\psi}_j}{|B_j|_{\lambda}}<\varepsilon$ for $j>N$.

As the sets $B_j^{\ast}$ are bounded, there exists a constant $C>0$ such that
\[
\sum_{j=1}^{N} \sum_{|\beta| \leq n} \int\limits_{B_j^{\ast}} \frac {w(y)^{|\beta|}}{[x,y]^{n+\lambda}} \, \mathrm{d}V_{\lambda}(y) \leq C
\]
for all $x\in {\bbB_n}$. Since the sequence $( f_k )_{k=1}^{\infty}$, as well as the sequences of the derivatives of $f_k$,
converge to $0$ uniformly on compact subsets, we may choose $M\in \mathbb{N}$ such that
\[|D^\beta f_k(y)| \leq \varepsilon\]
for all $k\geq M$, $|\beta | \leq n$, $y \in B_j^{\ast}$ and $j \leq N$.

For $k \geq M$ we get
\bea
\sum_{j=1}^{\infty} \left|T_{\psi} (\chi_{B_j} f_k)(x)\right|
& \lesssim &
C_{\psi} \sum_{j=1}^{N}  \sum_{|\beta| \leq n} \,
\int\limits_{B_j^{\ast}} \frac { w(y)^{|\beta|} |D^\beta f_k(y)|}{[x,y]^{n+\lambda}} \, \mathrm{d}V_{\lambda}(y)
\rowpl
\,\varepsilon \sum_{j=N+1}^{\infty}  \sum_{|\beta| \leq n} \,
\int\limits_{B_j^{\ast}} \frac { w(y)^{|\beta|} |D^\beta f_k(y)|}{[x,y]^{n+\lambda}} \, \mathrm{d}V_{\lambda}(y)
\nonumber
\\
& \lesssim &
C_{\psi}\varepsilon + \varepsilon\sum_{|\beta| \leq n} \,
\int\limits_{\bbB_n} \frac { w(y)^{|\beta|} |D^\beta f_k(y)|}{[x,y]^{n+\lambda}} \, \mathrm{d}V_{\lambda}(y)
.\notag
\eea
Arguing as in the proof of Theorem \ref{thm:mainthm} we find that $\|T_{\psi}  f_k\|_{p,\lambda} \lesssim\varepsilon$ for sufficiently large $k$.
\end{proof}

\begin{proof}[Proof of Corollary \ref{main-cor}]
As $P_\lambda^{\rm an}$ and $P_\lambda$ are orthogonal projections and obviously
$A_{\lambda}^p ({\bf B}_n) \subset b_{\lambda}^p(\bbB_{2n})$, we have
$P_\lambda^{\rm an} = P_\lambda^{\rm an}P_\lambda$ on $L_{\lambda}^2 = L_{\lambda}^2(\bbB_{2n})$. Let $f \in L_{\lambda}^1$ have compact support. Choose a sequence $(f_n)_{n \in \bbN}$ in $L_{\lambda}^2$ with the same support as $f$ such that $\|f-f_n\|_{1,\lambda} \to 0$ as $n \to \infty$. Then
\begin{equation} \label{eqn:L1estimate}
	\|P_{\lambda}(f-f_n)\|_{p,\lambda} \lesssim \|f-f_n\|_{1,\lambda}
	\ \, \text{and} \ \,
	\|P_{\lambda}^{\rm an}(f-f_n)\|_{p,\lambda} \lesssim \|f-f_n\|_{1,\lambda}
\end{equation}
by Lemma \ref{lem:L1boundedness} and the respective result for the analytic case, which can be proven verbatim. In particular,
\[P_\lambda^{\rm an}P_\lambda f = P_\lambda^{\rm an}P_\lambda(f-f_n) + P_\lambda^{\rm an}P_\lambda f_n = P_\lambda^{\rm an}P_\lambda(f-f_n) + P_\lambda^{\rm an}f_n.\]
By \eqref{eqn:L1estimate}, the former converges to $0$ and the latter converges to $P_\lambda^{\rm an}f$ as $n \to \infty$. Hence $P_\lambda^{\rm an}f = P_\lambda^{\rm an}P_\lambda f$ for $f \in L_{\lambda}^1$ with compact support. Now let $\psi \in L^1_{\rm loc}$ and $f \in A_\lambda^p$. Then
\[\sum\limits_{j = 1}^\infty T_\psi^{\rm an}(\chi_{B_j}f) = \sum\limits_{j = 1}^\infty P_\lambda^{\rm an}(\psi\chi_{B_j}f) = \sum\limits_{j = 1}^\infty P_\lambda^{\rm an}P_\lambda(\psi\chi_{B_j}f).\]
As the series $\sum\limits_{j = 1}^\infty P_\lambda(\psi\chi_{B_j}f)$ converges in $L_\lambda^p$ by Theorem \ref{thm:mainthm}, and $P_\lambda^{\rm an}$ is continuous, we obtain
\[\sum\limits_{j = 1}^\infty P_\lambda^{\rm an}P_\lambda(\psi\chi_{B_j}f) = P_\lambda^{\rm an}\left(\sum\limits_{j = 1}^\infty P_\lambda(\psi\chi_{B_j}f)\right) = P_\lambda^{\rm an}T_\psi f.\]
In short, $T_\psi^{\rm an} = P_\lambda^{\rm an}T_{\psi}|_{A_\lambda^p}$. Properties such as boundedness and compactness therefore transfer from the harmonic to the analytic setting.
\end{proof}

\section{An example.}
\label{sec9}

Sufficient conditions for the boundedness of Toeplitz operators with radial symbols in Bergman spaces
$A_\lambda^p({\bf B}_n)$ have been given for example in \cite{GKV}, where $p=2$, or in \cite{LT}, where $n=1$. All such results concern symbols $\psi$ belonging at least to $L^1({\bf B}_n)$. Modifying
an example first presented in \cite{TV2} we show here that there are radial symbols with arbitrarily fast 
growing modulus, still satisfying the sufficient condition \eqref{eqn:weakCarleson} of Theorem 
\ref{thm:mainthm}  and thus  inducing bounded Toeplitz operators. 

If $\lambda > -1$ and $\psi = \psi(r)$ is a radial symbol on the ball $\bbB_n$, then our condition \eqref{eqn:weakCarleson} for boundedness simplifies to
\begin{equation}
\sup\limits_{m \in \bbN_0} 2^{m (1+ \lambda)} \!\!\!\!\!\!\! \sup\limits_{\rho \in [1-2^{-m},1-2^{-m-1}]} \left|\int_{1-2^{-m}}^{\rho} r^{n-1} \psi(r) (1-r^2)^\lambda \, \mathrm{d}r\right| < \infty.  \label{9.new}
\end{equation}
Let $f : [1,\infty) \to \bbR$ be any continuous function with $\inf\limits_{x \in [1,\infty)} 
f(x)x^{\lambda} > 0$. Define
\[
g : [1,\infty) \to [0,\infty), \quad g(x) = \int_1^x \frac{f(y)}{y^{1 - \lambda}} \, \mathrm{d}y
\]
and note that $g'(x) = \frac{f(x)}{x^{1 -  \lambda}} > 0$ and $g(x) \to \infty$ as $x \to \infty$. In particular, $g$ is invertible. Define
\[
\psi(r) = r^{-n+1}(1-r^2)^{-\lambda} f((1-r)^{-1})\exp\big(i \pi g((1-r)^{-1})\big).
\]
It is plain that the modulus of $\psi$ can be made to grow arbitrarily fast as $r \to 1$, nevertheless,
we claim that $\psi$ satisfies \eqref{9.new} and thus by Theorem \ref{thm:mainthm}, $T_\psi: b_\lambda^p
\to b_\lambda^p$  is bounded for all $1 < p < \infty$.

With the substitution $s := g((1-r)^{-1})$, i.e., $1-r = 1/g^{-1}(s)$,  we get for the modulus of 
imaginary part (the real part is treated in the same way) 
\begin{align*}
&\left|{\rm Im} \Big( \int_{1-2^{-m}}^{\rho} r^{n-1} \psi(r)(1-r^2)^\lambda \, \mathrm{d}r \Big) \right|\\
= & \left|\int_{1-2^{-m}}^{\rho} (1-r)^{-1 + \lambda}g'((1-r)^{-1})\sin\big(\pi g((1-r)^{-1})\big) \, \mathrm{d}r\right|\\
= & \left|\int_{g(2^m)}^{g((1-\rho)^{-1})} \big( g^{-1}(s) \big)^{-1-\lambda}\sin(\pi s) \, \mathrm{d}s\right|.
\end{align*}
Let $a$ be the smallest integer such that $a > g(2^m)$ and $b$ the largest integer such that $g((1-\rho)^{-1}) > b$. Then, as $\frac{1}{g^{-1}(s)}$ is decreasing, we get
\begin{align*}
&\left|\int_{g(2^m)}^{g((1-\rho)^{-1})} \big( g^{-1}(s) \big)^{-1-\lambda} \sin(\pi s) \, \mathrm{d}s\right|\\
\leq & \left|\int_{g(2^m)}^{a} \big( g^{-1}(s) \big)^{-1-\lambda} \sin(\pi s) \, \mathrm{d}s\right|
\\
& + 
\left|\sum\limits_{k = a}^{b-1}(-1)^k\int_k^{k+1} \big( g^{-1}(s) \big)^{-1-\lambda} |\sin(\pi s)| \, \mathrm{d}s\right|\\
& + \left|\int_b^{g((1-\rho)^{-1})} \big( g^{-1}(s) \big)^{-1-\lambda}\sin(\pi s) \, \mathrm{d}s\right|\\
\leq & 3 \cdot 2^{-m(1+ \lambda)} ,
\end{align*}
since we have the lower estimate $g^{-1} (s) \geq 2^m$ on all integration intervals and the series
in the middle term is alternating with decreasing absolute values of the terms.
This shows that $T_\psi$ is bounded. Similarly, replacing $g$ by
\[g : [1,\infty) \to [0,\infty), \quad g(x) = \int_1^x f(y) y^\lambda \, \mathrm{d}y\]
and choosing $\psi$ as above, yields a compact Toeplitz operator $T_{\psi}$. These examples show that symbols of arbitrary growth can induce bounded or even compact Toeplitz operators.

%----------------------------------------------------------------
\bibliographystyle{amsplain}

\end{document}